\documentclass[11pt]{article}
\usepackage{amsmath}
\usepackage{amsfonts}
\usepackage{amsthm}
\usepackage{indentfirst}
\usepackage{graphicx}

\newtheorem{Lem}{Lemma\;}[section]
\newtheorem{The}{Theorem\;}[section]
\newtheorem{Pro}{Proposition\;}[section]

\renewenvironment{proof}{\noindent{\bf Proof.~}}{\hfill $\square$ \\ \indent}

\def\beq{\begin{equation}\displaystyle}
\def\eeq{\end{equation}}

\def\RR{\mathbb{R}}
\def\SS{\mathbb{S}}
\def\NN{\mathbb{N}}
\def\ds{\displaystyle}
\def\eps{\varepsilon}
\def\bar#1{{\overline #1}}
\def\pa{\partial}

\def\calM{{\cal M}}
\def\calD{{\cal D}}

\def\smes{{\cal S}_{\cal M}}

\begin{document}
\thispagestyle{empty}

\begin{center}
\large{\textbf{ On the hydrodynamical limit for a one dimensional }} \\ 
\large{\textbf{ kinetic model of cell aggregation by chemotaxis }}
\end{center}

\vspace {1.5cm}

\begin{center}
{\sc\large Fran\c{c}ois James}  \ {\small and}
 \ {\sc\large Nicolas Vauchelet}
\end{center}

\vspace{1,5cm}
\begin{center}
\begin{minipage}[t]{10cm}
\small{ \noindent \textbf{Abstract.}
The hydrodynamic limit of a one dimensional kinetic model describing 
chemotaxis is investigated. The limit system is a conservation law
coupled to an elliptic problem for which the macroscopic velocity 
is possibly discontinuous. Therefore, we need to work with measure-valued densities. After recalling a blow-up result
in finite time of regular solutions for the hydrodynamic model, 
we establish a convergence result of the solutions of the kinetic
model towards solutions of a problem limit defined thanks to the flux.
Numerical simulations illustrate this convergence result.

\medskip


\medskip

\noindent \textbf{Mathematics~Subject~Classification~(2000):}
92C17, 35L65.

}
\end{minipage}
\end{center}

\medskip


\bigskip

\section{Introduction}

\subsection{Modeling}

Chemotaxis is a process in which a population of cells rearrange its structures, reacting
to the presence of a chemical substance in the environment.
In the case of positive chemotaxis, cells migrate towards 
a concentration gradient of chemoattractant, allowing them
to aggregate.
Since several years, many attemps for describing chemotaxis from 
a Partial Differential Equations viewpoint have been considered. 
The population at the macroscopic level is described by a 
coupled system on its density and the chemoattractant concentration.
The most famous Patlak, Keller and Segel model \cite{KS,patlak}
is formed of parabolic or elliptic equations coupled through a
drift term.
Although this model has been successfully used to describe aggregation 
of cells, this macroscopic model has several shortcomings, 
for instance the detailed individual movement of cells is not 
taken into account.

In the 80's, experimental observations have shown that the 
motion of bacteria (e.g. {\it Escherichia Coli}) 
is due to the alternance of `runs and tumbles'
\cite{alt,erbanothmerbacterie,othdunalt,othste}.
Therefore kinetic approaches for chemotaxis have been proposed.
The so-called Othmer-Dunbar-Alt model \cite{othdunalt,othste,perthame}
describes the dynamic of the distribution function $f$ of cells 
at time $t$, position $x$ and velocity
$v$ and of the concentration of chemoattractant $S$~:
\beq\label{cinetique}
\left\{
\begin{array}{l}
\ds \pa_t f + v\cdot \nabla_x f = \int_{v'\in V} (T[S](v'\to v) f(v') 
- T[S](v\to v') f(v)) \,dv',    \\
\ds -\Delta S + S = \rho(t,x) := \int_{v\in V} f(t,x,v)\,dv.
\end{array}\right.
\eeq
In this equation, $V$ is the set of admissible velocities. The turning kernel $T[S](v' \to v)$ denotes the probability of cells to change
their velocity from $v'$ to $v$.
Several works have been devoted to the mathematical study of this kinetic
system, under various assumptions
on the turning kernel, see for instance \cite{chalubperth,bourncalv,erbanhwang,hwang}.
Here we shall assume that the velocities of cells have the same 
modulus $c$, so that $V= \SS_c:=\{ v \mid \|v\|=c\} $.

Derivation of macroscopic models from \eqref{cinetique} has been 
investigated by several authors. When the chemotactic orientation, or taxis, that is the weight of the turning kernel, is small
compared to the unbiased movement of cells, the limit equations are of diffusion or drift-diffusion type.
In \cite{hillenothmer, othhill}, 
the authors show that the Patlak-Keller-Segel model can 
be obtained as a diffusive limit for a given smooth chemoattractant
concentration. A rigorous proof for the case of a nonlinear coupling to
an equation for the chemical can be found in \cite{chalubperth}, leading to a drift-diffusion equation. 

In this paper we focus on the opposite case, where taxis instead of undirected 
movement is dominating. The model has been proposed in \cite{dolschmeis}, and we briefly recall
how it is obtained. The limit problem is usually of hyperbolic type, see for instance \cite{filblaurpert}.
Dominant taxis is reflected in the transport model by
the fact that the dominating part of the turning kernel depends
on the gradient of the chemoattractant. At this stage, two possible models are
encountered.
On the one hand, we can assume that cells are able to compare 
the present chemical concentrations 
to previous ones and thus to respond to temporal gradients along their 
paths. The decision to change direction and turn or to continue moving 
depends then on the concentration profile of the chemical $S$ along 
the trajectories of cells. 
Thus the turning kernel takes the form (independant on $v$)
\beq\label{Tdt}
T[S](v'\to v) = \phi(\pa_t S + v'\cdot \nabla_x S).
\eeq
On the other hand, if cells are large enough, it can be assumed that they are 
able to sense the gradient of the chemoattractant instantly so that
we can use instead the expression 
\beq\label{T}
T[S](v'\to v) = \phi(v'\cdot \nabla_x S).
\eeq
Theoretical results as well as numerical simulations for models \eqref{cinetique}--\eqref{Tdt} and \eqref{cinetique}--\eqref{T} 
are proposed in \cite{nv}.

The function $\phi$ in the preceding formul\ae\ is the turning rate, obviously it has to be positive and monotone. 
More precisely, for attractive chemotaxis, the turning rate is smaller if cells swim in a favorable 
direction, that is $\pa_t S + v'\cdot \nabla_x S\ge 0$ (or $v\cdot \nabla_x S\ge 0$). Thus $\phi$ should be a nonincreasing function.
The converse holds true for repulsive chemotaxis. 
A simplified model for this phenomenon is the following choice for
$\phi$: we fix a positive parameter $\alpha$, and take
\beq\label{phi}
\phi\in C^\infty(\RR), \quad \phi' \leq 0, \quad
\phi(x)= 
\left\{\begin{array}{ll}
\phi_0 & \quad\mbox{ if } x < -\alpha, \\
\phi_0/4 & \quad\mbox{ if } x > \alpha, \\
\end{array}\right.
\eeq
where $\phi_0$ is a given constant.
Moreover, for the sake of clarity of the paper, we will assume the
following symmetry on $\phi$: there exists an odd function 
$\widetilde{\phi}$ such that
\beq\label{phisym}
\phi(x)=\phi_0\big(\frac{5}{8} + \widetilde{\phi}(x)\big), \quad \mbox{ and }\quad
\widetilde{\phi}(-x) = -\widetilde{\phi}(x).
\eeq

The turning kernel \eqref{Tdt}, compared to \eqref{T}, makes drastic changes in the behaviour of the solutions to the kinetic 
model (see \cite{nv}). Up to now we cannot take it into account in the theory, so that we focus in the following on the expression \eqref{T}.
As observed above, this can be considered as a biologically relevant model.

In the turning kernel, a specific parameter quantifies the ``memory'' of the bacteria.
When this parameter is small, a specific asymptotic regime leads to a macroscopic, hydrodynamic model.
In order to introduce this parameter, we rescale the system \eqref{cinetique} by setting
$$
\begin{array}{c}
\ds x = x_0 \bar{x}, \qquad t=t_0 \bar{t}, \qquad v=v_0\bar{v},  \\[2mm]
\ds S(t,x)=S_0\bar{S}(\bar{t},\bar{x}), \qquad 
f(t,x,v)=f_0 \bar{f}(\bar{t},\bar{x},\bar{v}), \qquad
\phi(z)=\phi_0 \bar{\phi}(z),
\end{array}
$$
where $\phi_0$ is the typical value for the size of the turning kernel, 
$v_0=c$ is the typical speed, $x_0=L_x$ is the characteristic length 
of the device and the typical time is defined by $t_0=x_0/v_0$.
Dropping the bars, the scaled version of \eqref{cinetique} reads
\beq\label{nodimension}
\begin{array}{c}
\ds \pa_t f + v\cdot \nabla_x f = \frac{1}{\eps}
\int_V \left( T[S](v'\to v) f(v')
- T[S](v\to v') f(v)\right)\,dv',  \\[2mm]
-\Delta S + S = \rho,
\end{array}
\eeq
where $\ds \eps=\frac{v_0}{\phi_0 x_0} \ll 1$ is the parameter we are interested in: it corresponds 
to the time interval of information sampling for the bacteria. The hydrodynamic limit corresponds to
$\epsilon\to 0$, and we first recall formally how it is obtained.

\subsection{Formal hydrodynamic limit}
We focus in this work on the one dimensional version of \eqref{nodimension},
so that the transport takes place in $x\in\RR$ and the set of velocity is $V = \{-c,c\}$. 
The expression of the turning kernel simplifies in such a way that 
\eqref{nodimension} with \eqref{T} rewrites
\beq\label{cin1D}
\pa_t f_\eps + v \pa_x f_\eps = \frac{1}{\eps} (\phi(- v\pa_x S_\eps) f_\eps(-v)
- \phi(v\pa_x S_\eps) f_\eps(v)), \qquad v\in V.
\eeq
\beq\label{ellip1D}
-\pa_{xx}S_\eps + S_\eps = \rho_\eps = f_\eps(v) + f_\eps(-v).
\eeq
We formally let $\eps$ go to $0$ assuming that $S$ and $f$ admit a Hilbert expansion
$$
f_\eps = f_0 + \eps f_1 + \cdots, \qquad S_\eps = S_0 + \eps S_1 + \cdots
$$
Multiplying \eqref{cin1D} by $\eps$ and taking $\eps=0$, we find 
\beq\label{equilib}
\phi(- c\pa_x S_0) f_0(-c) = \phi(c\pa_x S_0) f_0(c).
\eeq
Summing equations $\eqref{cin1D}$ for $c$ and $-c$, we obtain :
\beq\label{moment}
\pa_t (f_\eps(c) + f_\eps(-c)) + c \pa_x (f_\eps(c) - f_\eps(-c)) = 0.
\eeq
Moreover, from equation \eqref{equilib} we deduce that
$$
f_0(c)-f_0(-c) = \frac{\phi(- c\pa_x S_0)-\phi(c\pa_x S_0)}
{\phi(- c\pa_x S_0)+\phi(c\pa_x S_0)} (f_0(c) + f_0(-c))
$$
The density at equilibrium is defined by $\rho := f_0(c) + f_0(-c) = \int f_0(v)\,dv$. 
Taking $\eps = 0$ in \eqref{moment} we finally obtain
$$
\pa_t \rho + \pa_x (a(\pa_x S_0) \rho) = 0,
$$
where $a$ is defined by
$$
a(\pa_x S_0) = c\, \frac{\phi(- c\pa_x S_0)-\phi(c\pa_x S_0)}
{\phi(- c\pa_x S_0)+\phi(c\pa_x S_0)}= \frac{4}{5} c\,(\phi(- c\pa_x S_0)-\phi(c\pa_x S_0)),
$$
and we have used \eqref{phisym} for the last identity.
Notice that $a$ is actually a macroscopic quantity, since we can rewrite 
$$
a(\pa_xS_0) = -\frac 45 \int_V v\,\phi(v\pa_xS_0)\,dv,
$$
so that this expression is independant of the sign of $c$.

We couple this equation with the limit of the elliptic problem \eqref{ellip1D} for the 
chemoattractant concentration, so that, in summary, and dropping the index $0$, the formal hydrodynamic limit is the
following system
\begin{eqnarray}
&&\ds \pa_t \rho + \pa_x (a(\pa_x S) \rho) = 0, 
\label{eqrhohydro} \\[2mm]
&&\ds a(\pa_xS)= \frac 45 c\,(\phi(-c\pa_xS)-\phi(c\pa_xS)),
\label{eqahydro} \\[2mm]
&&\ds -\pa_{xx} S +S = \rho, 
\label{eqShydro}
\end{eqnarray}
complemented with the boundary conditions
\beq\label{bordhydro}
\rho(t=0,x)=\rho^{ini}(x), \qquad
\lim_{x\to \pm \infty} \rho(t,x) = 0, \qquad 
\lim_{x\to \pm \infty} S(t,x)=0. 
\eeq

The formal hydrodynamic limit from \eqref{cin1D}--\eqref{ellip1D} to \eqref{eqrhohydro}--\eqref{eqahydro}--\eqref{eqShydro} has been
 obtained in \cite{dolschmeis} and proved rigorously in the two-dimensional setting for a given smooth $S$. 
The aim of this paper is to give an account of the problems and open questions arising in the study
of the whole coupled system. 

\subsection{Preliminary remarks}

First notice that, even in this one dimensional framework, this study leads to
 difficulties mainly due to the lack of uniform estimates
for the solutions to the kinetic model when $\eps$ goes to zero and consequently to the very weak regularity of the solutions 
to the limit problem. 
Even though existence of weak solutions to the kinetic model is ensured in a $L^p$ setting, no uniform 
$L^\infty$ bounds can be expected. The reader is referred to \cite{nv} for some numerical evidences of this phenomenon, 
which is the mathematical translation of the concentration of bacteria. This is some kind of ``blow-up in infinite time'',
which for $\eps=0$ leads to actual blow-up in finite time, and creation of Dirac masses.
Moreover the balanced distribution vanishing the right 
hand side of \eqref{cin1D} depends on $S_\eps$; thus the techniques 
developed e.g. in \cite{chalubperth} cannot be applied.

We turn now to formal considerations about the limit system, noticing on the one hand that  a solution of \eqref{eqShydro} has
 the explicit expression
 \beq\label{Sexplicit}
S(t,x) = K*\rho(t,.)(x),\quad\mbox{where }K(x) = \frac{1}{2}e^{-|x|},
\eeq
so that the macroscopic conservation equation for $\rho$ \eqref{eqrhohydro} can be rewritten
	\begin{equation}\label{AgregLike}
\pa_t \rho + \pa_x(a(\pa_x K * \rho) \rho) = 0.
	\end{equation}
When $a$ is the identity function, this is exactly the so-called aggregation equation, 
which has been studied
by several authors, see \cite{bertozzi1,bertozzi2,bertozzi3,laurent} and references therein. In particular, 
finite time blow-up is evidenced when the kernel $K$ is not smooth enough.

On the other hand, taking $\alpha = 0$ in the definition of $\phi$ \eqref{phi} and 
assuming that the chemoattractant concentration is increasing for $x<x_0$
and decreasing for $x>x_0$ (which is usually true when cells aggregate 
at the position $x_0$), we deduce that $\ds a(\pa_x S)=-\frac 35 c\mbox{ sgn}(x-x_0)$ which presents a
singularity at $x=x_0$. The conservation equation \eqref{eqrhohydro} becomes therefore a linear conservation
equation with a discontinuous compressive velocity field, and it is well known that the solution is a Dirac mass.
If $\alpha$ is positive, it turns out that a Dirac mass appears as well, after a finite time.

In summary, we have to deal in the limit system with some kind of weakly nonlinear conservation equation on the density $\rho$.
Indeed on the one hand the expected velocity field depends on $\rho$, but in a nonlocal way. On the other hand, this equation behaves like
linear equations with discontinuous coefficients, in the sense that it admits measure-valued solutions. Therefore a major difficulty
in this study will be to define properly the velocity field $a=a(\partial_xS)$ and the product $a\rho$.

The paper is organized as follows. 
In Section \ref{Sec.Blowup} we consider the aggregation-like equation \eqref{AgregLike}, and recall existence and uniqueness results
as well as the existence of a finite time for which $L^\infty$-weak solutions of 
\eqref{eqrhohydro}--\eqref{eqShydro} blow up.
In Section \ref{SecConvKinet}, we investigate the hydrodynamical limit 
of system \eqref{cin1D}--\eqref{ellip1D} and prove in particular that it gives rise to a somehow natural definition
of the flux in the conservation equation.
Some numerical simulations illustrating this result are furnished 
in Section \ref{sec.numeric}.
Finally, we end this work with some conclusions and remarks.

\section{Aggregation-like equation} \label{Sec.Blowup}

In this section, we consider the equation
\beq\label{eq:rho}
\left\{\begin{array}{l}
\ds \pa_t\rho + \pa_x(a(\pa_xK*\rho) \rho) = 0, \\
\ds \rho(t=0,\cdot)=\rho^{ini}.
\end{array}\right.
\eeq
where $K$ is given by \eqref{Sexplicit}. We assume that
\beq\label{reg.condinit}
0<\rho^{ini}\in L^1\cap L^\infty(\RR).
\eeq
When $a(x)\equiv x$, this equation is the so-called aggregation equation
(see e.g. \cite{bertozzi1,bertozzi2,bertozzi3,bodvel,laurent}). 
It is known that for singular $\pa_xK$, solutions blow up in
finite time. More precisely, we show the blow-up in finite time of
$L^\infty$ weak solutions. Most of the results presented in this section 
are obtained thanks to a straightforward adaptation of techniques 
developed in \cite{bertozzi1, bertozzi2,bodvel,laurent}. Therefore
some proofs are not detailed.

\subsection{Existence and uniqueness of local $L^\infty$-weak solution}
We prove in this section the local existence and uniqueness of a solution.
\begin{The}\label{exist.reg}
Let $\rho^{ini} \in L^1\cap L^\infty(\RR)$. Then there exists a $T>0$
such that there exists a unique weak solution $\rho$ to \eqref{eq:rho};
moreover $\rho\in C([0,T];L^1\cap L^\infty(\RR))$.
\end{The}
The proof is an adaptation of results in \cite{bertozzi1, bertozzi2, laurent}.
We first recall the definition of the characteristics for this system:
$X(s;x,t)$ is a solution of the ODE
\beq\label{carac}
\frac{dX}{ds}(s;x,t)=a(\pa_xK*\rho)(s,X(s)), \qquad
X(t;x,t)=x.
\eeq
Then we have the following representation of the solution of \eqref{eq:rho}:
\beq\label{rep.carac}
\rho(t,x)=\rho^{ini}(X(0;x,t)) \exp\left({ -\int_0^t \pa_x a(\pa_xK*\rho)(s,X(s-t;x,t))\,ds}\right).
\eeq
The proof of this theorem relies strongly on the following estimates:
\begin{Pro}\label{estim.rho}
Let $\rho^{ini}$ such as in \eqref{reg.condinit} and let $\rho$ be 
a solution of \eqref{eq:rho} on $[0,T]$. Then there exists $T>0$ such 
that for all $t\in [0,T]$, there exists a nonnegative constant $C$ such that 
$$
\|\rho(t,\cdot)\|_{L^1(\RR)} +\|\rho(t,\cdot)\|_{L^\infty(\RR)}
\leq C,
$$
where $C$ only depends on $\|\rho^{ini}\|_{L^1(\RR)}$ and 
$\|\rho^{ini}\|_{L^\infty(\RR)}$.
\end{Pro}

\begin{proof}
The $L^1$ estimate is an easy consequence of the mass conservation. 
Then,
$$
|\pa_t\rho+a(\pa_xK*\rho)\pa_x\rho|=|-\pa_x(a(\pa_xK*\rho)) \rho|
\leq 2 \|a'\|_\infty |\rho|^2.
$$
Integrating along the characteristics curves, we get
$$
\|\rho(t,\cdot)\|_{L^\infty} \leq \|\rho^{ini}\|_{L^\infty} + 2\|a'\|_\infty
\int_0^t \|\rho(s,\cdot)\|_{L^\infty}^2\,ds.
$$
We deduce that as long as $2\|a'\|_\infty \|\rho^{ini}\|_{L^\infty} t<1$,
$$
\|\rho(t,\cdot)\|_{L^\infty} \leq \frac{\|\rho^{ini}\|_{L^\infty}}
{1-2\|a'\|_\infty \|\rho^{ini}\|_{L^\infty} t}.
$$
We notice that $T$ should satisfies the bound 
$T<1/(2\|a'\|_\infty \|\rho^{ini}\|_{L^\infty})$.
\end{proof}

\noindent{\bf Proof of existence.~}
We do not detail the proof of the existence of solution which can
be deduced thanks to an adaptation of \cite{bertozzi1, bertozzi2, laurent}, 
where the study of an aggregation equation is proposed.
We just recall the main argument of the proof in the following steps~:
\begin{enumerate}
\item We construct a family of approximating solutions $(\rho_\eps)$
by solving \eqref{eq:rho} with initial data $\rho^{ini}*g_\eps$ where
$g_\eps$ is a mollifier.
\item We state uniform Lipschitz estimates in space and time 
on the sequences $(a(\pa_xK*\rho_\eps))_\eps$ and $(X_\eps)_\eps$ and
use the Arzel\`a-Ascoli Theorem to extract converging subsequence.
\item We pass to the limit in the representation \eqref{rep.carac}.
{\hfill $\square$ \\ \indent}
\end{enumerate}

\noindent{\bf Proof of uniqueness.~}
The idea of this proof is to use the quantity $S$. Since this idea will 
be developed for measure-valued solutions, we detail this proof.
Computations are done for regular solutions, nevertheless they
can be made rigorous by introducing a regularization 
and passing to the limit (see \cite{bertozzi1}).
Let us consider two classical solutions $\rho_1$ and $\rho_2$.
Denoting $a_i=a(\pa_xK*\rho_i)$ for $i=1, 2$, we have
\beq\label{eq1-2}
\pa_t (\rho_1-\rho_2) + \pa_x(a_1(\rho_1-\rho_2))
+ \pa_x(\rho_2(a_1-a_2))=0.
\eeq
Define $S(t,x):=(\pa_xK*(\rho_1-\rho_2)(t,\cdot))(x)$ 
which solves the problem 
\beq\label{eqS1-2}
-\pa_{xx} S+S = \rho_1-\rho_2, \quad \mbox{ on } \RR.
\eeq
We notice that when $t=0$, we have $S(0,x)=0$. From the weak 
formulation of equation \eqref{eq1-2} with the test function 
$S$, we have 
\beq\label{formfaible}
\int_0^t\int_\RR \pa_t(\rho_1-\rho_2) S \,dxds = I + II
\eeq
where 
$$
\begin{array}{l}
\ds I = \int_0^t\int_\RR a_1(\rho_1-\rho_2) S \,dxds,
\\[3mm]
\ds II = \int_0^t \int_\RR \rho_2(a_1-a_2) S \,dxds.
\end{array}
$$
For the term $I$, we have using \eqref{eqS1-2} and integration by parts
$$
I=\int_0^t\int_\RR a_1(-\pa_{xx} S+S) \, S \,dxds = \frac{1}{2}
\int_0^t\int_\RR \pa_x a_1 \,|\pa_xS|^2 \,dxds+
\int_0^t\int_\RR a_1 S^2 \,dxds,
$$
Moreover, 
$$
\pa_x a_1 = -\frac 45 c (\phi'(-c\pa_xS_1)+\phi'(c\pa_xS_1))
\pa_{xx} S_1 \leq \max\{\frac 85 c \|\phi'\|_{L^\infty} S_1,0\},
$$
where we use the fact that $\phi$ is a nonincreasing positive 
function. From the $L^\infty$-bound on $S_1$, we deduce that there 
exists $\beta\in L^1([0,T])$ such that
$\pa_x a_1 \leq \beta$. Thus 
$$
I \leq \frac 12 \int_0^t \beta(s) \int_\RR |\pa_xS|^2\,dxds+
\int_0^t\int_\RR a_1 S^2 \,dxds.
$$
Then, the estimate $|a_1|\leq \frac 35 c$ gives
\beq\label{estimI}
I \leq \frac 12 \int_0^t \beta(s) \|\pa_x S\|^2_{L^2}\,ds+ \frac{3}{5} c 
\int_0^t \|S\|_{L^2}^2\,ds.
\eeq
For the term II of \eqref{formfaible}, we have thanks to the Cauchy-Schwarz 
inequality
$$
\begin{array}{ll}
\ds |II|= & \ds \left|\int_0^t \int_\RR 
\rho_2(a(\pa_xK*\rho_1)-a(\pa_xK*\rho_2)) S \,dxds\right| \\[2mm]
&\ds \leq \int_0^t \|\rho_2\|_{L^\infty} \|a'\|_\infty
\|\pa_xK*(\rho_1-\rho_2)\|_{L^2} \|S\|_{L^2}\,ds.
\end{array}
$$
Since $\pa_xS=\pa_xK*(\rho_1-\rho_2)$, we obtain
\beq\label{estimII}
|II|\leq \frac 12 \int_0^t \|\rho_2\|_{L^\infty} \|a'\|_\infty
(\|\pa_xS\|_{L^2}^2+\|S\|_{L^2}^2)\,ds.
\eeq
Then, we notice that using \eqref{eqS1-2} and thanks to an integration 
by parts the left hand side of \eqref{formfaible} can be rewritten
\beq\label{lhs}
\int_0^t\int_\RR \pa_t(\rho_1-\rho_2) S \,dxds =
\frac 12 \int_\RR (|\pa_xS|^2+S^2)\,dx
\eeq
Finally, we deduce from \eqref{formfaible}, \eqref{estimI}, \eqref{estimII}
and \eqref{lhs},
$$
\int_\RR (|\pa_xS|^2+S^2)\,dx \leq \int_0^t\|\rho_2\|_{L^\infty} \|a'\|_\infty
\left(\beta(s) \|\pa_xS\|_{L^2}^2+ \frac 35 c\|S\|_{L^2}^2\right)\,ds.
$$
Uniqueness follows from a Gronwall type argument.
{\hfill $\square$ \\ \indent}

\subsection{Blow-up in finite time}

The blow-up of solutions of a one dimensional aggregation solution 
is proposed for instance in \cite{bodvel} where it is proved by the method
of characteristics that
aggregation of mass occurs. In \cite{bertozzi1, bertozzi2}, the
finite time blow-up is obtained thanks to an energy estimate.
We assume that the initial data is given symmetric with respect
to $0$ and positive. It is easy to show then that for all $t>0$,
$\rho(t,x)=\rho(t,-x)$. 
Moreover, for the sake of simplicity, we assume that there 
exists $\delta>0$ such that supp$(\rho^{ini}) \subset [-\delta,\delta]$.
Then for all $x>\delta$ the function $S^{ini}=K*\rho^{ini}$ satisfies
$\pa_xS^{ini}(x)<0$, so that the characteristics defined by \eqref{carac}
are inward. Thus for all $t>0$, supp$(\rho(t,\cdot)) \subset [-\delta,\delta]$.

The energy of the system is defined as
\beq\label{nrj}
E(t)=\frac 12 \int_\RR (|\pa_xS|^2+|S|^2) \,dx = \frac12 \int_\RR \rho S\,dx,
\eeq
where the last formulation is obtained by integration by parts.
On the one hand, we have the obvious bound
\beq\label{nrj.bound}
E(t)\leq \frac 12 \|\rho\|_{L^1} \|K*\rho\|_{L^\infty}
\leq \frac 12 \|\rho\|_{L^1}^2.
\eeq
On the other hand, using \eqref{eq:rho} we have
$$
\frac{d}{dt}E(t) = \int_\RR a(\pa_xS)\pa_xS \rho \,dx.
$$
Moreover $\|\pa_xS\|_{L^\infty}\leq \frac 12 \|\rho^{ini}\|_{L^1}$.
Since the function $a$ is assumed to be regular,
there exists $\zeta>0$ such that $a(x)x\geq \zeta |x|^2$ for all
$x\in [-\frac 12 \|\rho^{ini}\|_{L^1},\frac 12 \|\rho^{ini}\|_{L^1}]$.
Thus,
\beq\label{estim.dtnrj}
\frac {d}{dt}E(t) \geq \zeta \int_\RR |\pa_xS|^2 \rho\,dx.
\eeq
We now make use of the following result whose proof is given in 
\cite{bertozzi2}:
\begin{Pro}
There exists a constant $C>0$ such that 
for all $\delta$ sufficiently small, we have for any symmetric 
nonnegative function $\rho$ in $L^1(\RR)$ with a compact support in
$[-\delta,\delta]$,
$$
\int_\RR |\pa_xK*\rho|^2 \rho \,dx \geq C.
$$
\end{Pro}
Then, from \eqref{estim.dtnrj} there exists a constant $C>0$ such that
for all $t>0$,
$$
E(t)-E(0)\geq Ct.
$$
Therefore, with \eqref{nrj.bound} we have proved
\begin{The}
Let $\rho$ be a symmetric solution of \eqref{eq:rho} with  symmetric, positive initial 
data with compact support included in $[-\delta,\delta]$.
For sufficiently small $\delta$, there exists a time $T^*>0$ 
for which the solution $\rho$ ceases to exist, i.e.
$$
\lim_{t\to T^*} \|\rho(t,.)\|_{L^p(\RR)} = +\infty, \quad 
\mbox{ for }p\in (1,\infty).
$$
\end{The}

\section{Convergence for the kinetic model}\label{SecConvKinet}

In this section we investigate the convergence of a sequence of solutions to the  
microscopic model \eqref{cin1D}--\eqref{ellip1D}. We are not able yet to obtain rigorously  
\eqref{eqrhohydro}--\eqref{eqShydro}. We actually prove that the whole sequence of solutions is 
convergent, and that the macroscopic density satisfies a conservation equation with a uniquely determined flux.
More precisely, the main result of this section is the following theorem. We introduce the macroscopic densities
	$$
\rho_\eps = \int_V f_\eps\,dv,\qquad \rho = \int_V f\,dv.
	$$

\begin{The}\label{the.conv}
Let $T>0$ and let us assume that $\rho^{ini}$ is given in ${\cal M}_b(\RR)$.
Let $(f_\eps,S_\eps)$ be a solution to the kinetic--elliptic equation 
\eqref{cin1D}--\eqref{ellip1D} with initial data $f_\eps^{ini}$
such that $\rho_\eps^{ini}:=\int_V f_\eps^{ini}\,dv=\eta_\eps*\rho^{ini}$ 
where $\eta_\eps$ is a mollifier.
Then as $\eps\to 0$, the sequence converges to $(\rho, S)$ in the following sense\,:
$$
\begin{array}{c}
\ds \rho_\eps \rightharpoonup \rho
\qquad \mbox{ in } \quad \smes:=C([0,T];{\cal M}_{b}(\RR)-\sigma({\cal M}_{b},C_0)), \\[2mm]
\ds S_\eps \rightharpoonup S \qquad \mbox{ in } \quad
C([0,T];W^{1,\infty}(\RR))-weak,
\end{array}
$$
and $(\rho,S)$ is the unique solution in the distribution sense of
\beq\label{eq.flux}
\left\{
\begin{array}{l}
\ds \pa_t \rho + \pa_x J = 0, \\[2mm]
\ds -\pa_{xx} S + S = \rho,
\end{array}\right.
\eeq
complemented with initial data $\rho^{ini}$ and where 
$$
J = -\pa_x(A(\pa_x S)) + a(\pa_x S) S \quad \mbox{ a.e. }
$$
\end{The}

Before turning to the proof of this result we notice that the problem \eqref{eq.flux}
is equivalent to
\beq\label{eq:Slim}
\pa_tS -\pa_xK*\left[\pa_x(A(\pa_xS)) + a(\pa_xS)S\right] = 0, \quad 
\mbox{ in } \calD'(\RR).
\eeq
This is obtained by taking the convolution with $K$ of the first equation 
in \eqref{eq.flux}. 
This emphasizes the key role of $S$ in the study of the limit.

\subsection{Preliminary results}
First we recall the following statement on the kinetic-elliptic problem.
\begin{The}\label{th.cin1D}
Let $T>0$ and $\eps>0$. Assume $f_\eps^{ini}\in C(\RR)$.
Then problem \eqref{cin1D}--\eqref{ellip1D} complemented with initial 
data $f_\eps^{ini}$ admits a unique weak solution in 
$C([0,T]\times \RR\times V)\times C([0,T];C^2(\RR))$.
Moreover, we have the following estimates uniform in $\eps>0$\,:
\beq\label{momentf}
\int_\RR\int_V |v|^k f_\eps \,dxdv = |v|^k |\rho^{ini}|(\RR)\, ,\quad 
\mbox{ k }\in \NN. 
\eeq
\end{The}
\begin{proof}
The proof of the existence can be found in \cite{nv}. 
The estimates \eqref{momentf} rely on the conservation of the mass
and on the fact that since $v\in V=\SS_c$, $|v|$ is constant.
\end{proof}

Then, we furnish a convergence result for a sequence of functions $S$.
\begin{Lem}\label{convpaS}
Let $(\rho_n)_{n\in\NN}$ be a sequence of measures that converges weakly
towards $\rho$ in $\smes$. Let $S_n(t,x)=(K*\rho_n(t,\cdot))(x)$ and $S(t,x)=(K*\rho(t,\cdot))(x)$,
where $K$ is defined in \eqref{Sexplicit}. Then when $n\to\infty$ we have
	$$\begin{array}{rcl}
\pa_xS_n(t,x) &\longrightarrow & \pa_xS(t,x)\quad\mbox{ for a.e. } t\in [0,T],\ x\in \RR, \\
\pa_xS_n(t,x) &\rightharpoonup & \pa_xS(t,x)\quad\mbox{in }L^\infty w-* 
	\end{array}$$
\end{Lem}

\begin{proof}
We have that
\beq\label{defpaxSn}
\pa_xS_n(t,x) = \big(\pa_xK*\rho_n(t,\cdot)\big)(x) = \int_\RR \frac 12\frac{x-y}{|x-y|} e^{-|x-y|} \,\rho_n(t,dy).
\eeq
Let $\eps>0$, we regularize the convolution kernel by introducing the following 
functions~: 
$$
\phi_{x,\eps}(y)= \left\{
\begin{array}{ll}
\ds \frac 12\frac{x-y}{|x-y|} e^{-|x-y|}, \qquad \mbox{ on } (-\infty,x-\eps]\cup[x,+\infty), \\
\ds \frac 12\left(\frac{1+e^{-\eps}}{\eps} (y-x) +1\right),
\qquad \mbox{ on } (x-\eps,x). \\
\end{array}\right.
$$
$$
\psi_{x,\eps}(y)= \left\{
\begin{array}{ll}
\ds \frac 12\frac{x-y}{|x-y|} e^{-|x-y|}, \qquad \mbox{ on } (-\infty,x]\cup[x+\eps,+\infty), \\[2mm]
\ds \frac 12\left(\frac{1+e^{-\eps}}{\eps} (y-x) -1\right),
\qquad \mbox{ on } (x,x+\eps). \\[2mm]
\end{array}\right.
$$
With this definition, we clearly have for all $x$, $y$ in $\RR$
\beq\label{estimpsiphi}
\psi_{x,\eps}(y) \leq \frac 12\frac{x-y}{|x-y|} e^{-|x-y|} \leq \phi_{x,\eps}(y).
\eeq
Moreover by definition of the weak convergence,
$$
\lim_{n\to +\infty} \int_\RR \phi_{x,\eps}(y)\,\rho_n(t,dy) = 
\int_\RR \phi_{x,\eps}(y)\,\rho(t,dy).
$$
Then, from \eqref{defpaxSn} and \eqref{estimpsiphi}, we deduce
$$
\limsup_{n\to +\infty} \pa_xS_n \leq \int_\RR \phi_{x,\eps}(y)\,\rho(t,dy).
$$
Moreover,
$$
\begin{array}{ll}
\ds \int_\RR \phi_{x,\eps}(y)\,\rho(t,dy) &\ds =\pa_xS +\frac 12\int_{x-\eps}^x
\left(1+ \frac{1+e^{-\eps}}{\eps}(y-x)-e^{y-x}\right)\,\rho(t,dy)  \\[3mm]
&\ds \leq \pa_xS +\frac 12(1-e^{-\eps})|\rho(t,\cdot)|(\RR)
\end{array}
$$
By the same token with $\psi_{x,\eps}$, we obtain the estimate
$$
\begin{array}{ll}
\ds \pa_xS-\frac 12(1-e^{-\eps})|\rho(t,\cdot)|(\RR) & \ds \leq
\liminf_{n\to +\infty} \pa_xS_n \\[2mm]
&\ds  \leq \limsup_{n\to +\infty} \pa_xS_n 
\leq \pa_xS+\frac 12(1-e^{-\eps})|\rho(t,\cdot)|(\RR).
\end{array}
$$
Letting $\eps\to 0$, we get $\lim_{n\to +\infty} \pa_xS_n(t,x) = \pa_x S(t,x)$
for almost all $t\in [0,T]$ and $x\in \RR$.
\end{proof}
We turn now to the uniqueness of $S$, which is the key point to get the uniqueness result in Theorem \ref{the.conv}.

\subsection{Uniqueness for $S$}\label{uniqueS}
In \cite{bertozzi1} (see proof of Theorem \ref{exist.reg} too), 
the authors obtain the uniqueness on the
aggregation equation by introducing a quantity which appears here 
naturally to be the potential $S$. They get an estimate that relies strongly
on the $L^\infty$ bound on the density whereas here it is only measure-valued
with a finite total variation. Therefore, we have to work in a weaker space, and
 we use the fact that the function $S$ defined by $S=K*\rho$ is a weak solution of 
\eqref{eq:Slim}. We have the following result.
\begin{Pro}\label{uniqS}
Let $S_1$ and $S_2$ be two weak solutions of 
\eqref{eq:Slim} in 
$L^\infty([0,T];B(\RR))\cap C([0,T];W^{1,1}(\RR))$
 with initial data $S_1^{ini}$ and 
$S_2^{ini}$ respectively.
Then there exists a nonnegative constant $C$ such that 
$$
\|S_1-S_2\|_{L^\infty([0,T];W^{1,1}(\RR))} \leq 
C \|S_1^{ini}-S_2^{ini}\|_{W^{1,1}(\RR)}.
$$
\end{Pro}

\begin{proof}
We introduce a function $A\in C^\infty(\RR)$ such that $A'=a$. 
Now, differentiating \eqref{eq:Slim}
and noticing that $K$ satisfies $-\pa_{xx}K+K=\delta_0$, we get
\beq\label{eq:dSlim}
\pa_t\pa_xS +\pa_x(A(\pa_xS))-\pa_xK*A(\pa_xS) + K*(a(\pa_xS)S)-a(\pa_xS)S = 0.
\eeq
The definition of $S$, $S(t,x)=(K*\rho(t,\cdot))(x)$, implies that
$\pa_xS$ belongs to $L^\infty(0,T;BV(\RR))$. Therefore equations 
\eqref{eq:Slim}--\eqref{eq:dSlim} have a sense in their weak formulation.
Let $S_1$ and $S_2$ satisfy the weak formulations of 
\eqref{eq:Slim}--\eqref{eq:dSlim} with initial data $S_1^{ini}$ and 
$S_2^{ini}$ respectively.
We denote by $a_1=a(\pa_x S_1)$ and $a_2=a(\pa_xS_2)$. We deduce from \eqref{eq:dSlim} that
$$
\begin{array}{c}
\ds \pa_t\pa_x(S_1-S_2) + \pa_x(A(\pa_xS_1)-A(\pa_xS_2)) = \\[4mm]
\ds \pa_xK*(A(\pa_xS_1)-A(\pa_xS_2)) + a_1S_1-a_2S_2 - K*(a_1S_1-a_2S_2).
\end{array}
$$
Multiplying this equation by ${\rm sign}(\pa_x(S_1-S_2))$,
integrating with respect to $x$ and using the properties
of the convolution product, we deduce
$$
\begin{array}{ll}
\ds \frac{d}{dt} \int_\RR |\pa_x(S_1-S_2)|\,dx \leq & \ds 
\|\pa_xK\|_{\infty} \int_\RR |A(\pa_xS_1)-A(\pa_xS_2)|\,dx + \\[4mm]
&\ds +(1+\|K\|_{\infty}) \int_\RR |a_1S_1 - a_2S_2|\,dx.
\end{array}
$$
The function $a$ being regular, we deduce
\beq\label{borndxS12}
\frac{d}{dt} \int_\RR |\pa_x(S_1-S_2)|\,dx \leq C_0
\int_\RR |\pa_x(S_1-S_2)|\,dx + C_1 \int_\RR |S_1-S_2|\,dx.
\eeq
By the same token with equation \eqref{eq:Slim}, it leads to 
\beq\label{bornS12}
\frac{d}{dt} \int_\RR |S_1-S_2|\,dx \leq C_2
\int_\RR |\pa_x(S_1-S_2)|\,dx + C_3 \int_\RR |S_1-S_2|\,dx.
\eeq
Summing \eqref{bornS12} and \eqref{borndxS12}, we deduce that there 
exists a nonnegative constant $C$ such that
$$
\frac{d}{dt} \|S_1-S_2\|_{W^{1,1}(\RR)} \leq C\|S_1-S_2\|_{W^{1,1}(\RR)}.
$$
Applying the Gronwall Lemma allows to conclude the proof.
\end{proof}

\subsection{Proof of Theorem \ref{the.conv}}
Let $(f_\eps,S_\eps)$ be a solution of \eqref{cin1D}--\eqref{ellip1D}.
For fixed $\eps>0$, we have $f_\eps \in C([0,T]\times\RR\times V)$.
We define the flux $\ds J_\eps:=\int_V v f_\eps\,dv$ and the macroscopic velocity
$$
a(\pa_xS_\eps)=-\frac{\int_V v \phi(v\pa_xS_\eps)\,dv}
{\int_V \phi(v\pa_xS_\eps)\,dv}
= \frac 45 v (\phi(-v\pa_xS_\eps)-\phi(v\pa_xS_\eps)).
$$
We can rewrite the kinetic equation \eqref{cin1D} as
$$
\pa_t f_\eps + v\pa_xf_\eps= \frac{1}{\eps}(\phi(-v\pa_xS_\eps)\rho_\eps
-\frac 54 f_\eps). 
$$
Taking the zeroth and first order moments, we get
\begin{eqnarray}
&\ds \pa_t \rho_\eps + \pa_x J_\eps = 0,
\label{moment0} \\[2mm]
&\ds \pa_t J_\eps + v^2 \pa_x \rho_\eps = \frac{1}{\eps}\frac 54
 (a(\pa_xS_\eps)\rho_\eps -J_\eps).
\label{moment1}
\end{eqnarray}

 From \eqref{moment0}, we deduce that 
$$
\forall\, t\in [0,T], \quad |\rho_\eps(t,\cdot)|(\RR) = |\rho^{ini}|(\RR).
$$
Therefore, for all $t\in [0,T]$ the sequence $(\rho_\eps(t,\cdot))_\eps$
is relatively compact in $\calM_b(\RR)-\sigma(\calM_b(\RR),C_0(\RR))$.
Moreover, there exists $u_\eps\in L^\infty([0,T],BV(\RR))$ such that
$\rho_\eps=\pa_x u_\eps$. From \eqref{moment0}, we get that 
$\pa_tu_\eps=-J_\eps$ and with estimate \eqref{momentf} for $k=1$ we deduce that
$u_\eps$ is bounded in Lip$([0,T],L^1(\RR))$. It implies the equicontinuity 
in $t$ of $(\rho_\eps)_\eps$. Thus the sequence $(\rho_\eps)_\eps$ is 
relatively compact in $\smes$ and we can extract a subsequence still
denoted $(\rho_\eps)_\eps$ that converges towards $\rho$ in $\smes$.

We recall that $S_\eps(t,x)=(K*\rho_\eps(t,\cdot))(x)$ where $K(x)=\frac 12 e^{-|x|}$.
Denoting $S(t,x):=(K*\rho(t,\cdot))(x)$, since $\rho\in \smes$, we have
$S\in L^\infty([0,T];BV(\RR))$. From Lemma \ref{convpaS}, 
the sequence $(\pa_xS_\eps)_\eps$ converges in $L^\infty w-*$ and a.e. to $\pa_xS$ as $\eps$ goes to $0$. 

 From \eqref{moment0}--\eqref{moment1}, we have in the distribution sense
\beq\label{eqrhoeps1}
\pa_t\rho_\eps+\pa_x(a(\pa_xS_\eps)\rho_\eps) = 
\pa_x(a(\pa_xS_\eps)\rho_\eps -J_\eps)=
\frac 45 \eps \pa_x (\pa_tJ_\eps + v^2\pa_x\rho_\eps) = R_\eps.
\eeq
Now, for all $\psi\in C^2_c((0,T)\times\RR)$, we deduce from \eqref{momentf}
$$
\left|\int (\pa_tJ_\eps + v^2 \pa_x\rho_\eps)\pa_x\psi\,dxdt \right| 
\leq |v||\rho^{ini}|(\RR) \|\pa_t\pa_x\psi\|_{L^\infty} + 
|v|^2|\rho^{ini}|(\RR)\|\pa_{xx}\psi\|_{L^\infty}.
$$
This implies that the limit in the distribution sense of the right-hand 
side $R_\eps$ of \eqref{eqrhoeps1} vanishes.

On the one hand, multiplying equation \eqref{ellip1D} by $a(\pa_xS_\eps)$
and introducing the real-valued function $A$ such that $A'=a$, we get
\beq\label{fluxeps}
a(\pa_xS_\eps)\rho_\eps = -\pa_x(A(\pa_xS_\eps)) + a(\pa_xS_\eps)S_\eps,
\eeq
so that we can rewrite the conservation equation \eqref{eqrhoeps1} as follows, in ${\cal D}'(\RR)$ :
\beq\label{eqrhoeps}
\pa_t\rho_\eps +\pa_x\left(-\pa_x A(\pa_xS_\eps) + a(\pa_xS_\eps)S_\eps\right)
=\frac 45 \eps \pa_x (\pa_tJ_\eps + v^2\pa_x\rho_\eps).
\eeq

Taking the limit $\eps\to 0$ in the distribution sense of equation 
\eqref{eqrhoeps}, we get that in ${\cal D}'(\RR)$
\beq\label{eq:rholim}
\pa_t\rho +\pa_x\left(-\pa_x A(\pa_xS) + a(\pa_xS)S\right) = 0,
\eeq
where $S(t,x)=(K*\rho(t,\cdot))(x)$.
We recall that we have chosen the initial data such that 
$\rho_\eps^{ini}=\eta_\eps*\rho^{ini}$ where $\eta_\eps$ is a mollifier.
Therefore $\rho_\eps^{ini}\rightharpoonup \rho^{ini}$ in 
$\calM_b(\RR)-\sigma(\calM_b(\RR),C_0(\RR))$. 

On the other hand, as noticed above, $S$ satisfies \eqref{eq:Slim} 
and \eqref{eq:dSlim} in the distribution sense.
Proposition \ref{uniqS} above asserts that $S$ satisfying 
\eqref{eq:Slim}--\eqref{eq:dSlim} is unique. Thus $\rho$ is unique
since, if we assume that there exist $\rho_1$ and $\rho_2$ satisfying
\eqref{eq:rholim} in the distribution sense, then by the uniqueness of 
the solution of \eqref{eq:Slim}--\eqref{eq:dSlim}, we have that
$K*\rho_1=K*\rho_2$ which implies that $\rho_1=\rho_2$.
Finally, thanks to the uniqueness, all the sequence $\rho_\eps$ converges
to $\rho$ in $\smes$.

\section{Numerical simulations}\label{sec.numeric}

We illustrate the previous convergence result with some numerical simulations
of the problem \eqref{cin1D}--\eqref{ellip1D}. 
We discretize the kinetic equation thanks to a semi-lagrangian scheme and
the elliptic equation for $S$ is discretized with $P_1$ finite elements.
We refer the reader to \cite{nv} for more details on the numerical scheme.
Notice that letting $\epsilon$ go to 0 in the simulations is very difficult because
of the high numerical diffusivity of the scheme. 

We have chosen to present simulations with realistic numerical values.
For the bacteria {\it Escherichia Coli} the 
velocity is $c=20.\, 10^{-6}$ $m.s^{-1}$ and the density of cells is 
$n_0=10^{11}$ $m^{-1}$. The domain is assumed to be an interval of length 
$x_0=1\,cm$. The turning kernel is given by \eqref{T} with $\phi$
in \eqref{phi}. Due to the large value of $n_0$, the value of the 
parameter $\alpha$ should be very large to have an influence; thus
this parameter does not play a role in the dynamics of bacteria and for
the simulations we have fixed $\alpha = 1$.
We assume that the initial concentration of cells is 
a Gaussian centered in the middle of the domain. 
We run simulations with three different values for
$\phi_0$\,: $\phi_0 = 0.05$, $1$ and $20$ so that $\eps = v_0/(\phi_0 x_0)$
takes the values $10^{-4}$, $2.\,10^{-3}$ and $4.\,10^{-2}$.

In Figures \ref{timeevol} and \ref{compareps} we present evolution of the
density of cells with respect to the time and to $\eps$. 
We observe the aggregation of cells in the center of the domain which is 
the first step of the formation of a Dirac.
As $\eps\to 0$, the aggregation phenomenon is faster and the solution 
seems to converge to a Dirac.
We display the evolution of the gradient of the chemoattractant concentration
$\pa_xS$ in Figures \ref{comparS} and \ref{comparSeps}. 
A singularity in the center of the domain appears clearly.

\begin{figure}[ht!]
\begin{center}
\includegraphics[width=6cm]{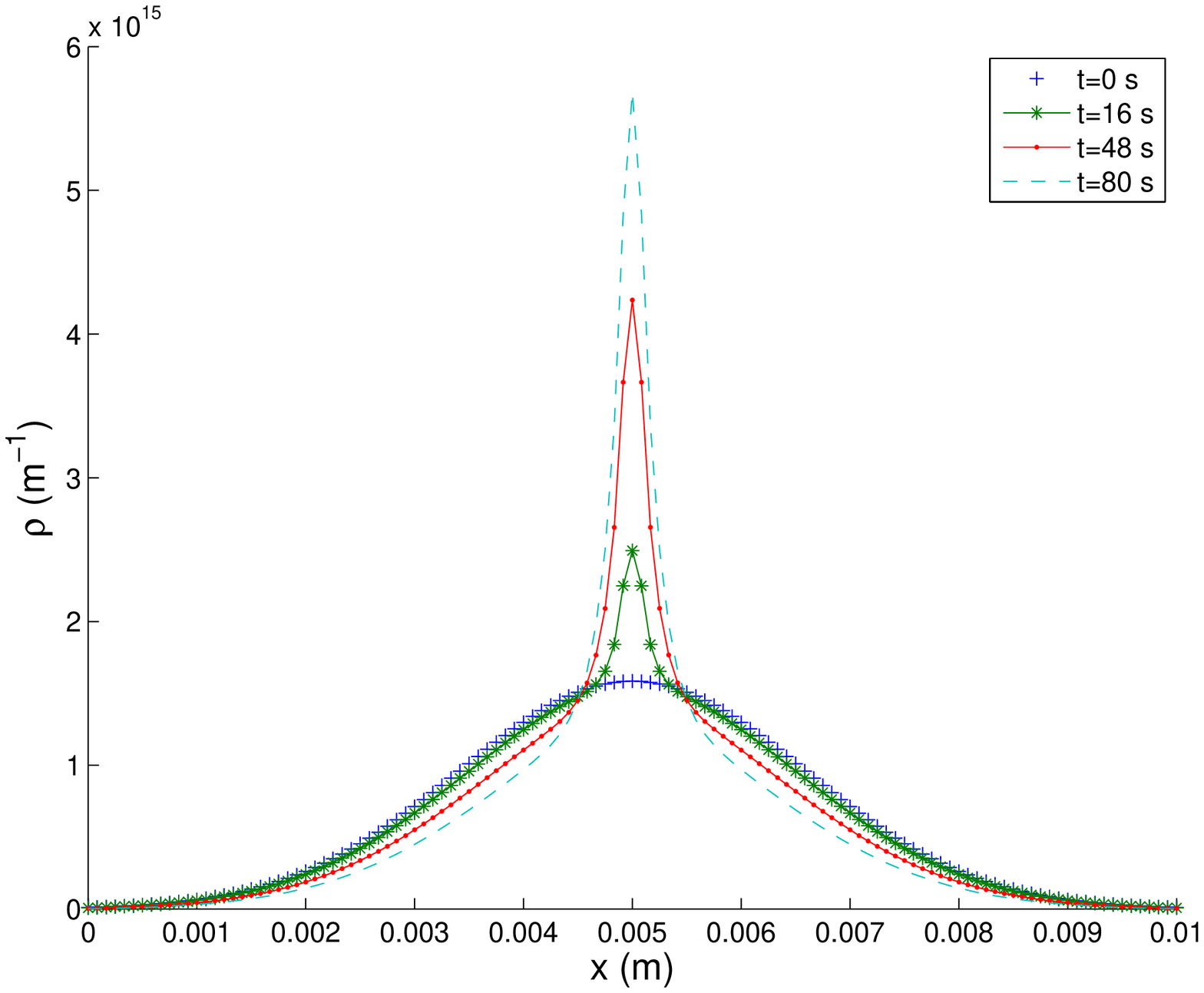}\hfill
\includegraphics[width=6cm]{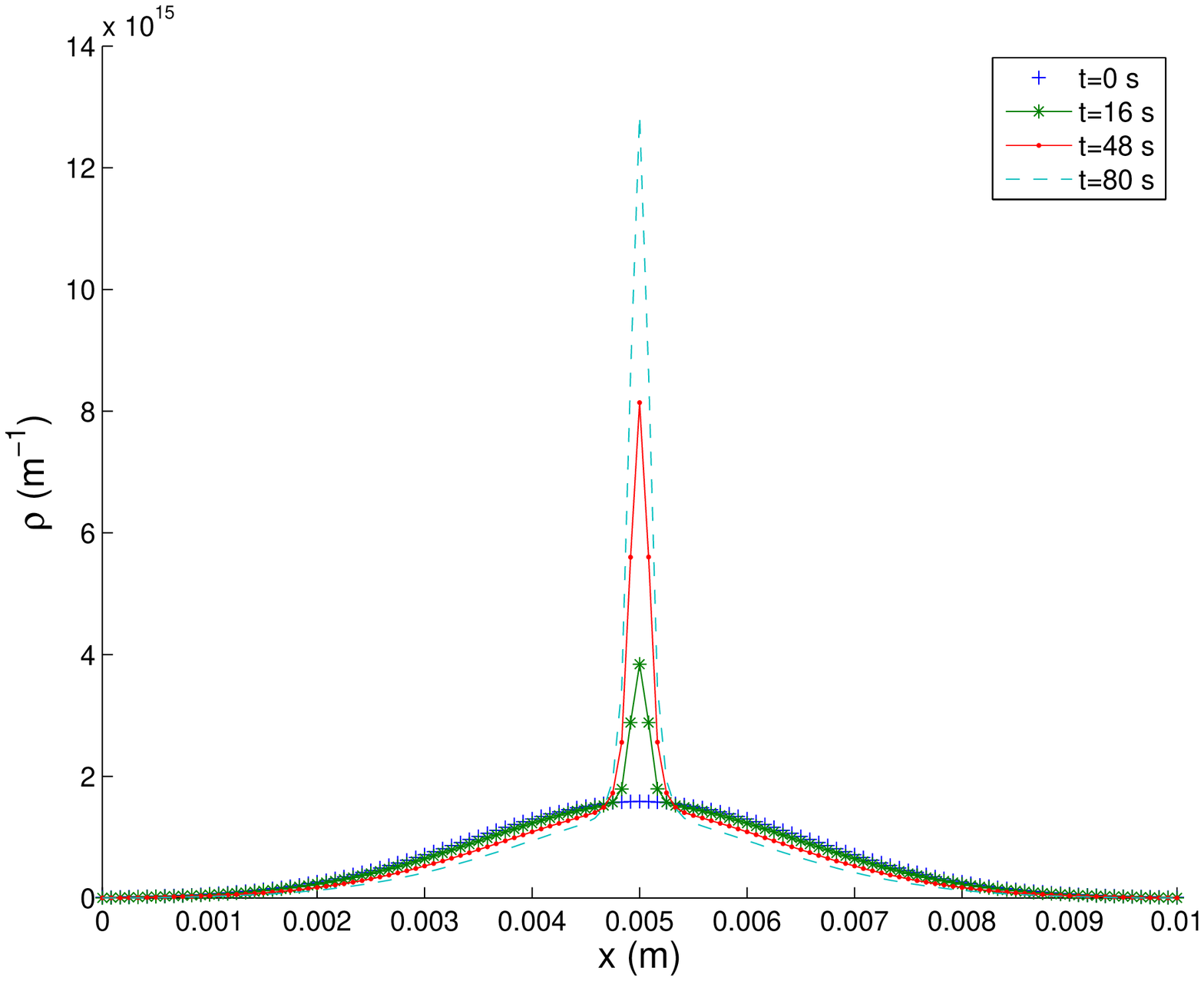}
\caption{Time evolution of the density $\rho$ of bacteria for different 
values of the parameters $\eps$. Left : $\eps=4.\,10^{-2}$. 
Right : $\eps =10^{-4} $.}
\label{timeevol}
\end{center}
\end{figure}

\begin{figure}[ht!]
\begin{center}
\includegraphics[width=10.2cm]{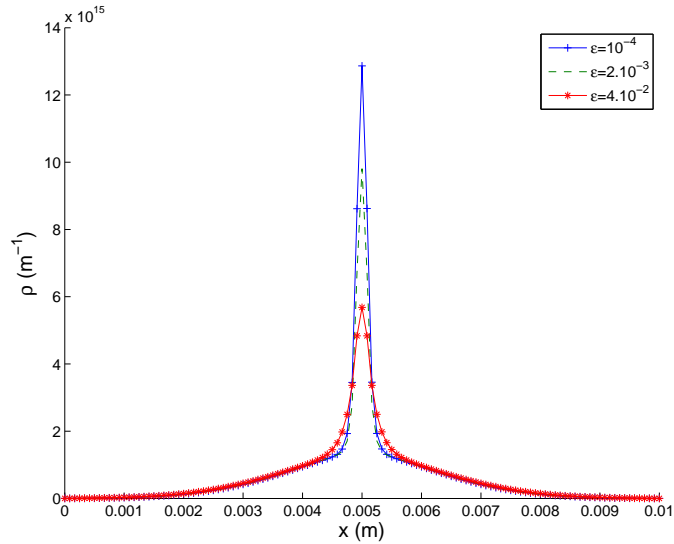}
\caption{Density $\rho$ of cells for different values of the parameters
$\eps$ at time $t=80$ s. As $\eps$ becomes smaller the concentration 
effect is more important.}
\label{compareps}
\end{center}
\end{figure}

\begin{figure}[ht!]
\begin{center}
\includegraphics[width=10.2cm]{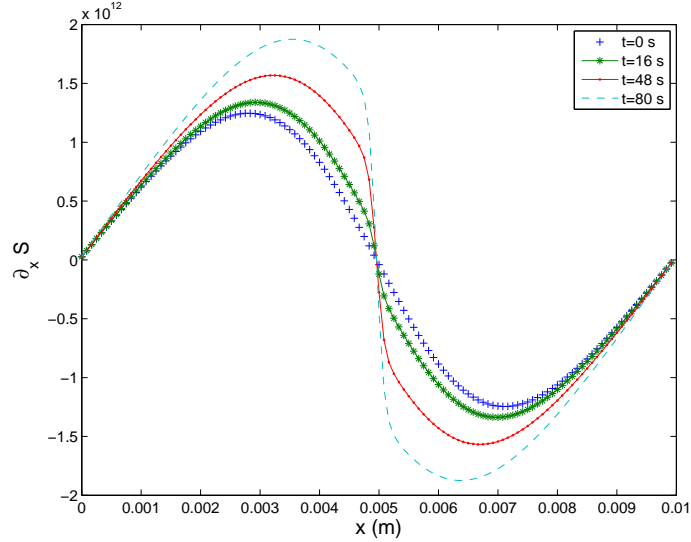}
\caption{Time dynamics of the gradient of potential $\pa_x S$ with 
$\eps=10^{-4}$. 
As time increases the derivative of the potential tends to become singular.}
\label{comparS}
\end{center}
\end{figure}

\begin{figure}[ht!]
\begin{center}
\includegraphics[width=10.2cm]{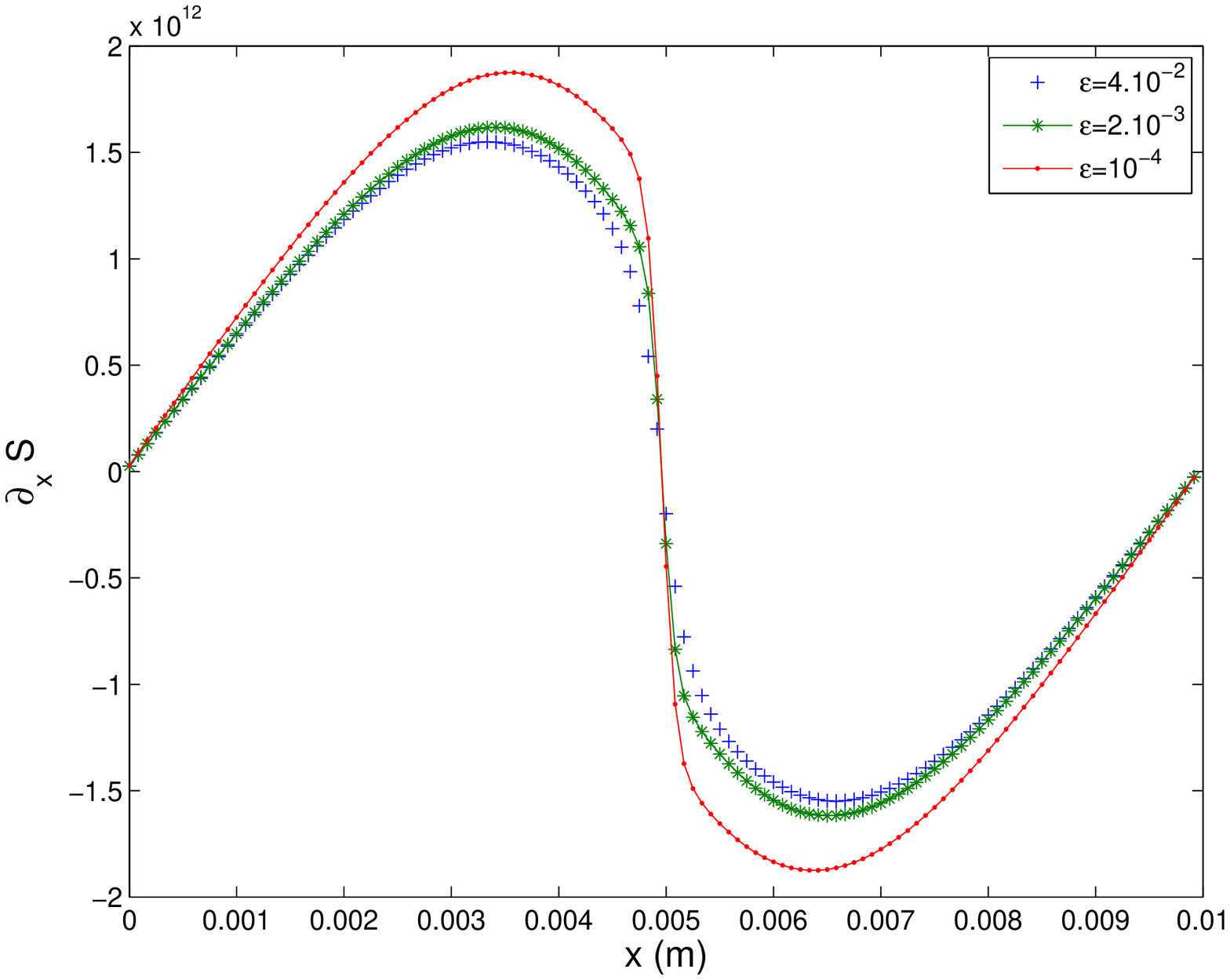}
\caption{Gradient of the potential $\pa_xS$ for different values for $\eps$
at time $t=80~$s.}
\label{comparSeps}
\end{center}
\end{figure}

\section{Conclusion}

In this work we have studied the convergence of a kinetic model 
of cells aggregation by chemotaxis towards a hydrodynamic model
which appears to be a conservation law coupled to an elliptic equation.
Although the limit of the macroscopic quantity $\rho_\eps$ and
$S_\eps$ have been obtained in Theorem \ref{the.conv},
this mathematical result is not completely satisfactory
since the limit model \eqref{eq.flux} does not 
allow to define a macroscopic velocity for the flux.
Formally, this macroscopic velocity is given by $a(\pa_xS)$ defined by 
\eqref{eqahydro}. However, since $\rho$ is only measure-valued, 
$\pa_xS$ belongs to $BV(\RR)$, hence we cannot give a sense to 
the product $a(\pa_xS) \rho$.

A possible convenient setting to overcome this difficulty is the notion 
of duality solutions, introduced by Bouchut and James \cite{bj1}.
In this framework, we can solve the Cauchy problem for conservation equations
in one dimension with a coefficient $a$ that satisfies a one-sided 
Lipschitz condition. The theory in higher dimensions is not complete 
\cite{bjm}, and Poupaud and Rascle \cite{pouras} (see also\cite{bianch}) for another approach,
which coincides with duality in the 1-d case.
It is actually not difficult to prove that $a$ defined in \eqref{eqahydro} 
is one-sided Lipschitz. In fact, from $\rho \geq 0$, we deduce
that $-\pa_{xx}S\leq S$. After straightforward computation, we get
$$
\pa_x (a(\pa_xS)) = -\frac 45 c (\phi'(-c\pa_xS)+\phi'(c\pa_xS))\pa_{xx}S.
$$
Therefore, $\phi$ being nonincreasing and smooth, we deduce
$$
\pa_x (a(\pa_xS)) \leq \max\{\frac 85 c \|\phi'\|_{L^\infty} S,0\}.
$$
And the properties of the convolution lead to
$$
\|S(t,\cdot)\|_{L^\infty} \leq \frac 12 |\rho(t,\cdot)|(\RR)=
\frac 12 |\rho^{ini}|(\RR).
$$
Finally, $a$ satisfies the OSL condition\,:
$$
\exists\, \beta \in L^1([0,T]), \quad \pa_x a(t,\cdot) \leq \beta(t)
\quad \mbox{in the distribution sense. }
$$

However, we are not able to prove the uniqueness of the duality solutions
for the hydrodynamic problem. In fact, the uniqueness proof in Section 
\ref{uniqueS} relies on the fact that the potential $S$ satisfies 
equation \eqref{eq:Slim} and thus on the definition the flux $J$
in \eqref{eq.flux}. In the framework of duality solution, the 
conservation equation is not a priori satisfied in the distribution 
sense. A generalized flux that has a priori no link with $J$ in
\eqref{eq.flux} is then introduced. The relation between these flux
and therefore the passage from $J$ to the macroscopic velocity $a$
is still an open question.

\vspace{0.5cm} \indent {\it Acknowledgments.}\; The authors
acknowledge Beno\^{\i}t Perthame for driving their attention on this problem of hydrodynamic limit.

\bigskip
\begin{center}

\end{center}

\bigskip
\bigskip
\begin{minipage}[t]{10cm}
\begin{flushleft}
\small{
\textsc{Fran\c{c}ois James}
\\*Universit\'e d'Orl\'eans, Math\'ematiques, Applications et Physique Math\'ematique d'Orl\'eans,
\\*CNRS UMR 6628, MAPMO
\\*F\'ed\'eration Denis Poisson, CNRS FR 2964,
\\* 45067 Orléans cedex 2, France
\\*e-mail: Francois.James@univ-orleans.fr
\\[0.4cm]
\textsc{Nicolas Vauchelet}
\\*UPMC Univ Paris 06, UMR 7598, Laboratoire Jacques-Louis Lions,
\\*CNRS, UMR 7598, Laboratoire Jacques-Louis Lions and
\\*INRIA Paris-Rocquencourt, Equipe BANG
\\*F-75005, Paris, France,
\\*e-mail: vauchelet@ann.jussieu.fr
}
\end{flushleft}
\end{minipage}

\end{document}